\documentclass[11pt,a4paper,notitlepage]{article}
\let\textacute\'
\usepackage{amsmath}
\usepackage[affil-it]{authblk}
\usepackage{amsthm}
\usepackage{amssymb}
\usepackage{latexsym}
\usepackage{longtable}
\usepackage{epsfig}
\usepackage{hhline}
\usepackage{amscd}
\usepackage{newlfont}
\usepackage[official]{eurosym}
\usepackage{enumerate}
\usepackage[page]{appendix}
\usepackage{geometry}
\usepackage{tabularx}
\usepackage{multirow}
\newtheorem{Lemma}      {Lemma} [section]
\newtheorem{Theorem}    [Lemma] {Theorem}

\newtheorem*{Conjecture}{The non-geometric Goldschmidt-Sims conjecture}
\newtheorem{Corollary}  [Lemma] {Corollary}

\theoremstyle{definition}

\numberwithin{equation}{section}
\newcommand{\core}{\mathop{\mathrm{core}}}

\newcommand{\Aut}{\mathop{\mathrm{Aut}}}

\newcommand{\Inn}{\mathop{\mathrm{Inn}}}

\newcommand{\Inndiag}{\mathop{\mathrm{Inndiag}}}

\newcommand{\ol}{\overline}
\numberwithin{equation}{section}

\newcommand\blfootnote[1]{%
  \begingroup
  \renewcommand\thefootnote{}\footnote{#1}%
  \addtocounter{footnote}{-1}%
  \endgroup
}

\begin{document}
\title{Some simplifications in the proof of the Sims conjecture}
\author{L\'aszl\'o Pyber\thanks{pyber@renyi.hu} }
\author{Gareth Tracey\thanks{gtracey@renyi.hu}}
\affil{Alfr\'ed R\'enyi Institute of Mathematics, Hungarian Academy of Sciences, Re\'altanoda utca
13-15, H-1053, Budapest, Hungary}

\maketitle

\begin{abstract}
\noindent We prove an elementary lemma concerning primitive amalgams and use it to greatly
simplify the proof of the Sims conjecture in the case of almost simple groups.    
\end{abstract}

\section{Introduction
\blfootnote{The work of the authors on the project leading to this application has received
funding from the European Research Council (ERC) under the European Unions Horizon 2020 research
and innovation programme (grant agreement No. 741420).}}
Over the last forty years or so,
the Classification of Finite Simple Groups (henceforth abbreviated to CFSG) has been greatly utilized to resolve a number of open questions and conjectures in finite group theory. One of the most well-known of these is the \emph{Sims conjecture}, first proposed by Charles Sims in \cite{Sims}. The conjecture states that if $G$ is a primitive permutation group, and $h>1$ is the length of a non-trivial orbit of a point stabilizer $H$ in $G$, then $H$ has order bounded above by a function of $h$. 

This conjecture was proved by P.J. Cameron, C.E. Praeger, J. Saxl, and G.M. Seitz in \cite{CPSS}. Therefore, it is now a theorem, and reads precisely as follows.
\begin{Theorem}\label{thm:Sims}
There exists a function $f:\mathbb{Z}\rightarrow \mathbb{R}$ such that whenever $G$ is a primitive permutation group, and $h>1$ is the length of a non-trivial orbit of a point stabilizer $H$ in $G$, then $|H|\le f(h)$.
\end{Theorem}
In this note, we present a simplified proof of Theorem \ref{thm:Sims}. Our proof does still require the CFSG, but the more involved methods used in \cite[page 502 and Section 4]{CPSS} are not required.

As a by-product of our approach, we also make progress on a more general version of the Sims conjecture, which we will refer to as the \emph{non-geometric Goldschmidt-Sims conjecture}, following Goldschmidt's work in \cite{Gold}.

To state this, we first need to introduce the language of \emph{primitive amalgams}. A triple of finite groups ($H,M,K$), usually written $H> M<K$,  is called a \emph{primitive amalgam} if $M$ is a subgroup of both $H$ and $K$ and:
\begin{enumerate}[(i)]
    \item Whenever $A$ is a normal subgroup of $H$ contained in $M$, we have $N_K(A)=M$; and 
    \item whenever $B$ is a normal subgroup of $K$ contained in $M$, we have $N_H(B)=M$.
\end{enumerate}
The non-geometric Goldschmidt-Sims conjecture reads as follows:
\begin{Conjecture}
There exists a function $f:\mathbb{Z}\times\mathbb{Z}\rightarrow\mathbb{R}$ such that whenever $H>M<K$ is a primitive amalgam with $|H:M|=h$ and $|K:M|=k$, then $|M|\le f(h,k)$.
\end{Conjecture}
We remark that the conjecture as stated above is referred to as the `baby Goldschmidt-Sims problem' by Fan in \cite{FanJA}. The general `Goldschmidt-Sims problem' seeks to classify all primitive amalgams, usually under certain additional conditions (for example, the case where $M$ is a $p$-group and $|H:M|$, $|K:M|$ are both prime and distinct from $p$ is dealt with in \cite{FanJA}). We also remark that if ($H,M,K$) is a primitive amalgam, then there is an associated graph $\Gamma=\Gamma(H,M,K)$ which encodes many properties of the amalgam. The `geometric Goldschmidt-Sims conjecture' is then a statement about this graph. See \cite{Fan, Gold} for more details.

Notice that if $G$ is a primitive permutation group, and $G_{\alpha}$ and $G_{\beta}$ are point stabilizers in $G$, then $G_{\alpha}>G_{\alpha}\cap G_{\beta}<G_{\beta}$ is a primitive amalgam. Thus, the non-geometric Goldschmidt-Sims conjecture really is a generalization of the Sims conjecture.

A number of papers (see \cite{FanS,Fan,Knapp,Thompson,Weiss,Wie,VanBon}) have addressed and/or made progress on the non-geometric Goldschmidt-Sims conjecture. In particular a result of of Wielandt (published by Knapp in \cite[Theorem 2.1]{Knapp}), extending an earlier result of Thompson \cite[Main Theorem]{Thompson}, implies the so-called \emph{Thompson--Wielandt theorem} (see Theorem \ref{thm:Thompson} below) which asserts that if $H>M<K$ is a primitive amalgam, then there exists a prime $p$ such that $|H:O_p(H)|$ can be bounded in terms of $|H:M|$ and $|K:M|$. We remark that 
\cite[Theorem 2.1]{Knapp} is only stated in the case where $H$ and $K$ are maximal subgroups in a finite group $G$, and $M:=H\cap K$ is core-free in $G$. It can be readily seen, however, that the proof works in the more general setting of primitive amalgams stated in Theorem \ref{thm:Thompson} below.  

Our contribution to the non-geometric Goldschmidt-Sims conjecture reads as follows.
\begin{Theorem}\label{thm:Gold}
There exists a function $f:\mathbb{Z}\times\mathbb{Z}\rightarrow\mathbb{R}$ such that whenever $H>M<K$ is a primitive amalgam with $|H:M|=h$ and $|K:M|=k$, and either $H$ or $K$ is abelian by $(h,k)$-bounded, then $|H|\le f(h,k)$ and $|K|\le f(h,k)$.
\end{Theorem}

We remark that in Theorem \ref{thm:Gold}, and throughout the paper, the terminology \emph{$X$ is $(n_1,\hdots,n_r)$-bounded}, for natural numbers $X,n_1,\hdots,n_r$, means that $X$ can be bounded by a function of $n_1,\hdots,n_r$. All other notation used in the paper is standard, though for a finite simple group $X$ of Lie type, we will use $\Inndiag(X)$ to denote the subgroup of $\Aut(X)$ generated by the group $\Inn(X)$ of inner automorphisms of $X$, together with the set of diagonal automorphisms of $X$. That is, $\Inndiag(X)=\langle \Inn(X),\alpha\text{ : }\alpha\text{ a diagonal automorphism of }X\rangle$.

\section{Proofs of Theorems \ref{thm:Sims} and \ref{thm:Gold}}
We begin preparations toward the proof of Theorem \ref{thm:Gold}, and the simplified proof of Theorem \ref{thm:Sims}, with the Thompson--Wielandt theorem \cite{Knapp,Thompson} mentioned in the introduction. 
\begin{Theorem}\label{thm:Thompson}
Let $H>M<K$ be a primitive amalgam, with $|H:M|=h$ and $|K:M|=k$. Then there exists a prime $p$ such that $p$, $|H:O_p(H)|$, and $|K:O_p(K)|$ are all $(h,k)$-bounded.
\end{Theorem}
\noindent We remark that the proof of the Thompson--Wielandt theorem does not require the CFSG.

We now discuss the proof of the Sims conjecture in \cite{CPSS}. The proof has five steps:
\begin{enumerate}[(1)]
    \item Reduction to the case where the primitive group $G$ in question is almost simple with socle $X$ a finite group of Lie type.
    \item Proof that it suffices to bound $|O_p(H)\cap\Inndiag(X)|$ in terms of $h$, where $p$ is as in the Thompson--Wielandt theorem.
    \item Proof that $|H|$ is $h$-bounded if $X$ has characteristic $p$.
    \item Proof that $|H|$ is $h$-bounded if $X$ is a classical group in $p'$ characteristic.
    \item Proof that $|H|$ is $h$-bounded if $X$ is a finite group of Lie type in $p'$ characteristic, and of $h$-bounded (Lie) rank.
\end{enumerate}
Steps (1), (3), and (4) make use of classical results in group and representation theory, and the theory of algebraic groups. We will not repeat the details here, and instead refer the reader to \cite[Section 1 and first three paragraphs of Section 2]{CPSS} for Step (1); the two paragraphs in \cite[Section 3]{CPSS} for Step (3); and \cite[Section 4a]{CPSS} for Step (4).

Steps (2) and (5) are much more involved, and rely on delicate results of Seitz \cite{Seitz} from the theory of algebraic groups (see page 502 and Section 4 of \cite{CPSS} for the details). For this reason, our aim in this note is to replace the proofs in Steps (2) and (5) with much simpler arguments.

Our next result, part (i) of which is Theorem \ref{thm:Gold}, will allow us to do so. We note that within its proof, we frequently use the standard fact that if $A$, $B$, and $C$ are subgroups of a group $G$ with $C$ a finite index subgroup of $B$, then $|A\cap B:A\cap C|\le |B:C|$. We will also use the notation $J(X)$ for the \emph{Thompson subgroup} of the finite group $X$. That is, $J(X):=\langle A<X\text{ : }A\text{ an abelian subgroup of }X\text{ of maximal order}\rangle$. We remark that the Thompson subgroup is usually only defined for $p$-groups, but our proofs here require the more general version.
\begin{Theorem}\label{thm:Gold0}
Let $H>M<K$ be a primitive amalgam with $|H:M|=h$ and $|K:M|=k$.
\begin{enumerate}[(i)]
    \item Suppose that $H$ is abelian by $(h,k)$-bounded. Then $|H|$, and hence $|K|$, is $(h,k)$-bounded.
    \item Suppose that $H$ and $K$ are subgroups of a finite group $G$, and that $I$ is a normal subgroup of $G$. If
    $H\cap I$ and $K\cap I$ are abelian by $(h,k)$-bounded, then the exponents of $H\cap I$ and $K\cap I$ are $(h,k)$-bounded.
\end{enumerate}
\end{Theorem}
\begin{proof}
Adopt the notation of the statement of the theorem, and set $N:=H$, $L:=K$, in case (i), and $N:=H\cap I$, $L:=K\cap I$, in case (ii).

Let $A$ be an abelian subgroup of $N$ of maximal order. Then $|N:A|$ is $(h,k)$-bounded, since $N$ is abelian by $(h,k)$-bounded. Suppose now that $K\cap A$ is core-free in $K$. Then $K$ is a transitive permutation group of degree $|K:K\cap A|= |K:K\cap N||K\cap N:K\cap A|\le |K:K\cap N||N:A|$.
If we are in case (i), then $N=H$ and $|K|$ can be bounded in terms of $|K:K\cap H||H:A|\le k|H:A|$. Thus, $|K|$, and hence $|H|$, is $(h,k)$-bounded. If we are in case (ii), then $\core_K(L\cap A)=1$. Thus, $K$ is a permutation group of degree $|K:L\cap A|$. The groups induced by the normal subgroup $L$ on each of its orbits are all permutation groups of degree $|L:L\cap A|$, and are all isomorphic to each other. Hence, the exponent of $L$ is $|L:L\cap A|$-bounded. But $|L:L\cap A|=|L:L\cap N||L\cap N:L\cap A|=|I\cap K:I\cap M||L\cap N:L\cap A|\le |K:M||N:A|=k|N:A|$. Thus, $L$ has $(h,k)$-bounded exponent. Since $|N:N\cap L|=|I\cap H:I\cap M|\le h$, it follows that $N$ has $(h,k)$-bounded exponent. 

Thus, we may assume that $C:=\core_K(K\cap A)$ is non-trivial. Since $A$ is abelian and $C\le A$, we have $A\le N_H(C)$, and since the amalgam is primitive, we have $N_H(C)=M$. Thus, $A\le M$. Hence, $J(N)$ is contained in $M\le H$. It follows that $J(H)=J(M)$ in case (i), and $J(N)=J(M\cap I)$ in case (ii). The same argument with $H$ replaced by $K$ (and $N$ replaced by $L$ in case (ii)), however, also shows that $J(K)=J(M)$ in case (i), and $J(L)=J(M\cap I)$ in case (ii). But then we have, in either case, that $J(N)=J(L)$ is normal in both $H$ and $K$ - a contradiction. This completes the proof.    
\end{proof}
\noindent We remark that the proof of Theorem \ref{thm:Gold0} does not require the Thompson--Wielandt theorem. Also, part (i) above proves Theorem \ref{thm:Gold}.

We can now replace Steps (2) and (5) in the proof of the Sims conjecture with the following easy consequences of Theorem \ref{thm:Gold0}(i).
We begin with Step (2):
\begin{Corollary}\label{cor:Step2}
Let $G$ be an almost simple primitive permutation group, with point stabilizer $H$, and let $h>1$ be the length of a non-trivial $h$-orbit. Suppose that the socle $X$ of $G$ is a finite group of Lie type, and let $p$ be the prime from the Thompson--Wielandt theorem. If $|O_p(H)\cap \Inndiag(X)|$ is $h$-bounded, then $|H|$ is $h$-bounded.
\end{Corollary}
\begin{proof}
Since $\Aut(X)/\Inndiag(X)$ has a normal subgroup of order at most $6$ with cyclic quotient, the Thompson--Wielandt theorem and the hypothesis of the corollary guarantee that $H$ is cyclic by $h$-bounded. Theorem \ref{thm:Gold0}(i) then yields the result. 
\end{proof}

Next, Step (5):
\begin{Corollary}\label{cor:Step5}
Let $G$ be an almost simple primitive permutation group, with point stabilizer $H$, and let $h>1$ be the length of a non-trivial $h$-orbit. Let $p$ be the prime from the Thompson--Wielandt theorem, and suppose that the socle $X$ of $G$ is a finite group of Lie type in $p'$ characteristic, and that $X$ has $h$-bounded (Lie) rank. Then $|H|$ is $h$-bounded.
\end{Corollary}
\begin{proof} 
Write $h=|H:H\cap H^g|$, for an element $g$ of $G$, and set $K:=H^g$, $M:=H\cap K$, and $I:=\Inndiag(G)$. Also, we may write $X=O^{p'}(\ol{X}_{\sigma})$ for a simple algebraic group $\ol{X}$ of adjoint type and a Steinberg endomorphism $\sigma$ of $\ol{X}$.
Since $p$ is coprime to the defining characteristic of $X$, \cite[II, Theorem 5.16]{SSt} implies that $O_p(H)\cap I$ normalizes of a maximal torus $\ol{T}$ in $\ol{X}$. Then $(O_p(H)\cap I)/(O_p(H)\cap I\cap \ol{T})$ is a subgroup of the Weyl group $W$ of $X$. Since the Lie rank of $X$ is $h$-bounded, the order of the group $W$ is $h$-bounded. We deduce that (a) $O_p(H)\cap I$ is abelian by $h$-bounded; (b) $O_p(H)\cap I$ has $h$-bounded subgroup rank; and (c) $O_p(H)\cap I$ has bounded derived length. 

It follows from (a) and the Thompson--Wielandt theorem that $H\cap I$ and $K\cap I=(H\cap I)^g$ are abelian by $h$-bounded.
Since $H>M<K$ is a primitive amalgam, Theorem \ref{thm:Gold0}(ii) then implies that $H\cap I$ and $K\cap I$ have $h$-bounded exponents. This, together with (b) and (c) above, implies that $|O_p(H)\cap I|$ and $|O_p(K)\cap I| $ are $h$-bounded. The result now follows from Corollary \ref{cor:Step2}.
\end{proof}
This completes our simplified proof of the Sims conjecture.

\end{document}